\documentclass[11pt, a4paper]{article}
\usepackage{amssymb,amsmath,amsfonts,amsthm}
\usepackage{enumitem}
\usepackage{fullpage}
\usepackage{calc}

\usepackage{tikz-cd}
\usepackage{pgfplots}
\usepackage{mathtools}
\pgfplotsset{compat=1.17}
\usetikzlibrary{math,calc,patterns}
\usepackage{tkz-euclide}

\usepackage{hyperref}

\numberwithin{equation}{section}

\newtheorem{thm}{Theorem}[section]
\newtheorem{lemma}[thm]{Lemma}
\newtheorem{prop}[thm]{Proposition}

\theoremstyle{definition}

\theoremstyle{remark}
\newtheorem{rmk}[thm]{Remark}

\newcommand\N{{\mathbb{N}}}
\newcommand\R{{\mathbb{R}}}
\newcommand\Z{{\mathbb{Z}}}

\newcommand\cA{{\mathcal{A}}}
\newcommand\cB{{\mathcal{B}}}
\newcommand\cF{{\mathcal{F}}}
\newcommand\cG{{\mathcal{G}}}

\newcommand\cW{{\mathcal{W}}}

\newcommand\brF{{\breve{F}}}
\newcommand\brY{{\breve{Y}}}
\newcommand\brmu{{\breve{\mu}}}
\newcommand\brpi{{\breve{\pi}}}

\newcommand\bff{{\bar{f}}}
\newcommand\bm{{\bar{m}}}

\newcommand\bY{{\bar{Y}}}

\newcommand\bF{{\bar{F}}}
\newcommand\bT{{\bar{T}}}
\newcommand\bM{{\bar{M}}}

\newcommand\bPP{{\bar{\PP}}}

\newcommand\bDelta{{\bar{\Delta}}}
\newcommand\balpha{{\bar{\alpha}}}
\newcommand\bmu{{\bar{\mu}}}
\newcommand\bpi{{\bar{\pi}}}

\newcommand\btau{{\bar{\tau}}}

\newcommand\tg{{\tilde{g}}}
\newcommand\ts{{\tilde{s}}}
\newcommand\tX{{\widetilde{X}}}

\newcommand\hg{{\hat{g}}}
\newcommand\hT{{\widehat{T}}}

\newcommand\eps{{\varepsilon}}

\DeclareMathOperator{\diam}{diam}

\DeclareMathOperator{\E}{\mathbb{E}}
\DeclareMathOperator{\Lip}{Lip}
\DeclareMathOperator{\PP}{\mathbb{P}}

\title{Rates in almost sure invariance principle
for nonuniformly hyperbolic maps}

\author{
C.~Cuny{\footnote{Universit\'e de Brest, LMBA, UMR CNRS 6205. Email: christophe.cuny@univ-brest.fr}},
J.~Dedecker{\footnote{\hbadness=5600 Universit\'e Paris Cit\'e,  Laboratoire MAP5 (UMR 8145). Email: jerome.dedecker@u-paris.fr}},
A.~Korepanov{\footnote{Loughborough, UK. Email: a.korepanov@lboro.ac.uk}},
F.~Merlev\`ede{\footnote{ LAMA,  Univ Gustave Eiffel, Univ Paris Est Cr\'eteil, UMR 8050 CNRS. Email: florence.merlevede@univ-eiffel.fr}}
}

\date{28 June 2024}

\begin{document}

\maketitle

\begin{abstract}
    We prove the Almost Sure Invariance Principle (ASIP) with close to optimal error
    rates for nonuniformly hyperbolic maps. We do not assume exponential contraction
    along stable leaves, therefore our result covers in particular slowly mixing invertible
    dynamical systems as Bunimovich flowers, billiards with flat points
    as in Chernov and Zhang~(2005) and Wojtkowski'~(1990) system of two falling balls.

    For these examples, the ASIP is a new result, not covered by prior works for various
    reasons, notably because in absence of exponential contraction along stable leaves,
    it is challenging to employ the so-called Sinai's trick (Sinai 1972, Bowen 1975)
    of reducing a nonuniformly hyperbolic system to a nonuniformly expanding one.

    Our strategy follows our previous papers on the ASIP for nonuniformly expanding maps,
    where we build a semiconjugacy to a specific renewal Markov shift and adapt the
    argument of Berkes, Liu and Wu~(2014). The main difference is that now the Markov shift is
    \emph{two-sided}, the observables depend on the full trajectory, both the future and the past.
\end{abstract}

\section{Introduction}

Let $T \colon M \to M$ be a nonuniformly hyperbolic dynamical system in the sense of Young~\cite{Y98}.
Notable examples of such systems are Axiom A or H\'enon-type attractors and various chaotic billiards.
Suppose that $T$ preserves an ergodic physical (Sinai-Ruelle-Bowen) probability measure $m$.
Let $\varphi \colon M \to \R$ be a H\"older observable with $\int \varphi \, dm = 0$ and
let $S_n = \sum_{j =0}^{n-1} \varphi \circ T^j$ be its corresponding Birkhoff sums.

It is common to expect that $S_n$, considered as a random process on the probability space $(M, m)$,
behaves like a Brownian motion. A standard result, proved for a large class of nonuniformly hyperbolic
maps under nonrestrictive assumptions, is the \emph{Almost Sure Invariance Principle} (ASIP), see
Melbourne and Nicol~\cite{MN09} or Gou\"ezel~\cite{G10}.
It holds if, without changing the distribution, the process $(S_n)_{n \geq 0}$ can be defined on a  probability space
supporting a Brownian motion $(W_t)_{t \geq 0}$ with variance $c^2$, such that with some $\delta \in (0, 1/2)$,
\begin{equation}
    \label{eq:ASIP}
    S_n = W_n + o(n^\delta)
    \quad \text{almost surely.}
\end{equation}
In this case we say that the process $S_n$ satisfies the ASIP with rate $o(n^\delta)$ and variance $c^2$. 

The ASIP has a range of useful implications, such as the (functional) central limit theorem
or the (functional) law of iterated logarithm, see Berkes and Philipp~\cite{BP79}.
The error rates received a significant amount of attention, see e.g.\ our paper~\cite{CDKM18} for an overview.
See also Zaitsev~\cite{Z13} for a historical overview for sums of independent random vectors.

There are natural nonuniformly hyperbolic dynamical systems where the ASIP is expected but is not covered
by  existing proofs for one reason or another. Prime examples are
Bunimovich flowers~\cite{B73} and Wojtkowski' system of two falling balls~\cite{W90}.
\begin{equation}
    \label{eq:MAPS}
    \begin{tikzpicture}[baseline=(current  bounding  box.center)]
        \tkzDefPoint(0,-1){A}
        \tkzDefPoint(-1,0){B}
        \tkzDefPoint(0,1){C}
      \tkzDefPoint(1,0){D}
        \tkzDefPoint(-1,-1){AB}
        \tkzDefPoint(-1,1){BC}
        \tkzDefPoint(1,1){CD}
        \tkzDefPoint(1,-1){DA}
        \tkzCircumCenter(A,B,AB)\tkzGetPoint{ABc}
        \tkzCircumCenter(B,C,BC)\tkzGetPoint{BCc}
        \tkzCircumCenter(C,D,CD)\tkzGetPoint{CDc}
        \tkzCircumCenter(D,A,DA)\tkzGetPoint{DAc}
        \tkzDrawArc[thick, lightgray](ABc,B)(A)
        \tkzDrawArc[thick, lightgray](BCc,C)(B)
        \tkzDrawArc[thick, lightgray](CDc,D)(C)
        \tkzDrawArc[thick, lightgray](DAc,A)(D)
        \filldraw [black] (-1,-1) circle (0.05);
        \draw [-stealth]  (-1,-1) -- (-0.3,-0.6);
        \node at (0,-1.75) {Flower billiard};
        \fill[pattern=north east lines] (6,-1.1) rectangle (8,-1);
        \draw [-]  (6,-1) -- (8,-1);
        \draw [dashed, lightgray]  (7,-1) -- (7,1.5);
        \filldraw [black] (7,-0.5) circle (0.25);
        \node at (6.4,-0.5) {$M$};
        \filldraw [black] (7,1.0) circle (0.25);
        \node at (6.4,1.0) {$m$};
        \draw [-stealth]  (7.75,1.0) -- (7.75,0.25) node[midway,right] {$g$};
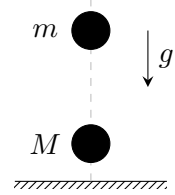
        \node at (7,-1.75) {Two falling balls};
    \end{tikzpicture}
\end{equation}
The nonuniformly hyperbolic structure for these maps, under certain natural assumptions,
is established by B\'alint, Borb\'ely and Varga~\cite{BBNV12} and Chernov and Zhang~\cite{CZ05a}.
Although the ASIP has not been proven, other statistical properties such as
the functional central limit theorem or (iterated) moment bounds are known,
see Fleming-V\'azquez \cite{F22} and Melbourne and Varandas~\cite{MV16}.

Another prototypical example is the intermittent baker's map~\cite[Example~4.1]{MV16}.
It is defined as a transformation of the unit square $M = [0,1] \times [0,1]$ by
\begin{equation}
    \label{eq:LSV2}
    T (x,y) =
    \begin{cases}
        (g(x), g^{-1}(y))       &   x \in [0, 1/2] , \\
        (2x - 1, (y+1) / 2  )     &   x \in (1/2, 1] ,
    \end{cases}
\end{equation}
where $g(x) = x ( 1 + 2^\alpha x^\alpha)$ and $\alpha \in (0,1)$.

In this paper we prove the ASIP for a class of nonuniformly hyperbolic maps,
including~\eqref{eq:MAPS} and~\eqref{eq:LSV2}, with rates close to optimum (see Theorem \ref{thm:ASIPNUH}).
Applied to the above examples, our main result is:

\begin{thm}
    \label{thm:examples}
    Suppose that $T \colon M \to M$ is either~\eqref{eq:LSV2} or 
    one of the maps~\eqref{eq:MAPS} under the assumptions of~\cite{BBNV12,CZ05a}.
    Let $m$ be the unique physical invariant probability measure.
    Let $v \colon M \to \R$ be H\"older with $\int v \, dm  = 0$ and $S_n = \sum_{j < n} v \circ T^j$.
    Then, for each $\eps > 0$, 
    \begin{enumerate}[label=(\alph*)]
        \item $S_n$ satisfies the ASIP with rate $o(n^{1/3} (\log n)^{4/3 + \eps})$ for both maps~\eqref{eq:MAPS}.
            (Possibly, the logarithmic factor can be reduced, see Theorem~\ref{thm:ASIPNUH} and Remark~\ref{rmk:BZ23}.)
        \item $S_n$ satisfies the ASIP with rate $o(n^\alpha (\log n)^{\alpha + \eps})$
            for the map~\eqref{eq:LSV2} when $\alpha < 1/2$.
    \end{enumerate}
\end{thm}

Our proofs are based on:
\begin{itemize}
    \item A construction of an extension of $T$, which is similar to a Young tower in the sense that
        it is (topologically) Markov with a tower structure and that the semiconjugacy map is Lipschitz,
        but with the difference that our extension is also Markov from the measure-theoretical point of view
        (i.e.\ can be described by a matrix of transition probabilities).
    \item A representation of the process $( v \circ T^j)_{ j \geq 0}$, without changing its distribution (but on another probability space), as   
        \begin{equation}
            \label{eq:Sng}
        \big (  \psi( \ldots, g_{k-1}, g_k, g_{k+1}, \ldots) \big )_{k \in {\mathbb Z}}
            ,
        \end{equation}
        where $(g_n)$ is a particular Markov chain and $\psi$ is a sufficiently regular function of its trajectories.
        A key tool in the study of the Markov chain is the \emph{meeting time} which is, informally, the time when two
        independent copies of the Markov chain first meet. Moments of the meeting time are used to control
        approximations of $\psi( \ldots, g_{j-1}, g_j, g_{j+1}, \ldots)$ by functions of finitely many
        coordinates.
    \item An adapted argument of Berkes, Liu and Wu~\cite{BLW14} which allows to deal with functions of the
        whole trajectory of a Markov chain instead of a sequence of independent identically distributed
        random variables as in~\cite{BLW14}.
    \end{itemize}

This is a continuation and extension of our previous works on nonuniformly expanding maps~\cite{CDKM18,CDKM20SD}.

\begin{rmk}
    As in~\cite{CDKM18,CDKM20SD}, our proof of ASIP happens on the level of a Markov chain~\eqref{eq:Sng}.
    At the same time, in~\cite{CDKM18,CDKM20SD} we had to deal with functions $\psi$ which depend only on
    future trajectories, i.e.\ $\sum_{j=0}^{n-1} \psi(g_j, g_{j+1}, \ldots)$.
    Working with nonuniformly hyperbolic dynamical systems, we necessarily have to deal with functions of
    the whole past and future trajectories. This creates problems which go beyond simple modifications of
    our previous arguments both in the reduction to the Markov chain (Section~\ref{sec:MST})
    and in the probabilistic part (Section~\ref{sec:ASIPNUH}). This can be contrasted with the proof of the ASIP
    in~\cite{BLW14} which is written for a one-sided Bernoulli shift rather than for a two-sided one
    ``to simplify the notation''.
\end{rmk}

\begin{rmk}
    It is curious that the degenerate case of zero variance is not covered by our general arguments and
    requires a special and very different treatment (Section~\ref{Sectvariance is zero}).
\end{rmk}

\begin{rmk}
    There are prior ideas which may be useful in finding alternative proofs of the ASIP for dynamical systems
    such as~\eqref{eq:MAPS} and~\eqref{eq:LSV2}, including the generalization to vector-valued observables.
    Notably:
    \begin{itemize}
        \item Melbourne and Nicol~\cite{MN05,MN09} prove the ASIP for vector-valued H\"older observables on
            nonuniformly hyperbolic dynamical systems with uniform contraction along stable leaves.
            Their proofs work using the so-called Sinai's trick (after Sinai~\cite{S72} and Bowen~\cite{B75},
            see~\cite[Introduction]{MV16}), representing a H\"older observable
            $v \colon M \to \R$ as $v = u + \chi \circ T - \chi$, where $\chi$ is bounded and $u$ is H\"older
            and constant on stable leaves.
            Then one can work with Birkhoff sums of $u$ on a quotient system which is nonuniformly expanding,
            and can be in turn reduced to a Gibbs-Markov map by inducing.
            For the ASIP, Melbourne and Nicol adapt the results of Philipp and Stout~\cite{PS75}
            and Kuelbs and Philipp~\cite{KP80}.

            In the absence of uniform contraction along stable leaves, as is the case for the maps~\eqref{eq:MAPS}
            and~\eqref{eq:LSV2}, Sinai's trick works but $\chi$ and $u$ lose regularity and become challenging
            to control. Nevertheless, some control is possible.
            Melbourne and Varandas~\cite{MV16} prove that $\chi,u \in L^p$ with some $p > 1$, and
            furthermore $u = m + \psi \circ T - \psi$ where again $m,\psi \in L^p$ and
            Birkhoff sums of $m$ are a reverse martingale.
        \item Gou\"ezel~\cite{G10} proves the ASIP for vector-valued random variables under a condition of
            certain exponential memory loss for characteristic functions. This condition does not hold
            for the maps~\eqref{eq:MAPS} or~\eqref{eq:LSV2}, but for instance Gou\"ezel~\cite[Theorem~2.4]{G10}
            uses it to prove the ASIP for a class of $L^p$ observables on Gibbs-Markov maps,
            and as a corollary on nonuniformly expanding maps via inducing.
    \end{itemize}
    However, using the above techniques, the error rates would be $O(n^{-1/4} (\log n)^\alpha)$ for some $\alpha>0$
    at best, whereas our approach is the only currently available one with better rates.
\end{rmk}

The paper is organized as follows. In Section~\ref{sec:MST} we state the assumptions on the
nonuniformly hyperbolic maps and prove that they can be modelled using Markov chains of a particular type. 
We call such systems Markov shift towers. In Section~\ref{sec:ASIPNUH} we prove the ASIP for H\"older observables
on nonuniformly hyperbolic maps through that on Markov shift towers.
Theorem~\ref{thm:examples} is a corollary of the more general Theorem~\ref{thm:ASIPNUH}.

We use the notation $\N = \{0,1,\ldots\}$ and $\N_0 = \{1,2,\ldots\}$.

\section{Nonuniformly hyperbolic maps as Markov shift towers}
\label{sec:MST}

In this section we give a standard definition of nonuniformly hyperbolic maps,
introduce Markov shift towers and show that every nonuniformly hyperbolic map
$T \colon M \to M$ with  invariant measure $m$
has an extension $f \colon \Delta \to \Delta$ with  invariant measure $\PP$
which is a Markov shift tower, and the semiconjugacy $\pi \colon \Delta \to M$ is Lipschitz.
\[
    \begin{tikzcd}[row sep=2.5em]
        \Delta \arrow[r,"f"] \arrow[d,"\pi"] & \Delta \arrow[d,"\pi"] & &
        \PP \arrow[d,mapsto,"(\pi)_*"]
        \\
        M \arrow[r,"T"] & M & &
        m
    \end{tikzcd}
\]
We also quantify the \emph{return time tails} of $f$ depending on those of $T$.
This result is stated formally in Theorem~\ref{thm:NUH}, but first we need a few pages of notation.

\subsection{Notation and result}
\label{sec:NUH}

\subsubsection{Nonuniformly hyperbolic maps}
\label{sec:NUHd}

Let $(M,d_0)$ be a bounded metric space with a Borel probability measure $m$ such that $(M,m)$ is a Lebesgue space.
Let $T \colon M \to M$ be measurable for the Borel sigma-field.
We assume that $T$ preserves $m$, is ergodic and  is nonuniformly hyperbolic in the following sense: 

\begin{itemize}
    \item There is a measurable subset $Y \subset M$, $m(Y) >0$, with an at most countable partition $\alpha$ and
        a return time $\tau \colon Y \to \{1,2,\ldots\}$ that is constant on each $a \in \alpha$ with value $\tau(a)$
        and $T^{\tau(y)}(y) \in Y$ for all $y \in Y$.
        We denote by $F \colon y \mapsto T^{\tau(y)}(y)$ the induced map.
        There is an $F$-invariant probability measure $\mu$ on $Y$ such that $\int \tau \, d\mu < \infty$ and which agrees with $m$ in the usual for induced maps sense:   
                \[
            m = \Bigl( \int \tau \, d\mu \Bigr)^{-1} \sum_{a \in \alpha} \sum_{k=0}^{\tau(a) - 1} (T^k)_* \mu_a
            ,
        \]
        where $\mu_a$ is the restriction of $\mu$ on $a$.

    \item Coding of orbits under $F$ by elements of $\alpha$ is non-pathological in the sense that
        the set of
        $(\ldots, a_{-1}, a_0, a_1, \ldots) \in \alpha^\Z$,
        for which there exists $(\ldots, y_{-1}, y_0, y_1, \ldots) \in Y^\Z$ such that
        $y_n \in a_n$ and $F(y_n) = y_{n+1}$, is measurable in $\alpha^\Z$ with the
        product topology and Borel sigma-algebra ($\alpha$ is equipped with the discrete topology). 
    \item For $y,z \in Y$ we define the separation time $s(y,z)$ as the least $n \geq 0$ such
        that $F^n(y)$ and $F^n(z)$ belong to different elements of $\alpha$,
        or $\infty$ if such $n$ does not exist.
        There are constants $K\ge 1$ and $\gamma\in(0,1)$ such that:
        \begin{itemize}
            \item If $y,z \in Y$, then for all $j \geq 0$,
                \begin{equation}
                    \label{eq:NUH:aa}
                    d_0(F^j(y),F^j(z)) \leq K (\gamma^{s(y,z) - j} + \gamma^{j})
                    .
                \end{equation}
            \item If $y,z \in a$, $a \in \alpha$, then for all $0 \leq j \leq \tau(a)$,
                \begin{equation}
                    \label{eq:NUH:ab}
                    d_0(T^j y,T^j z) \leq K \bigl[d_0(y,z) + d_0(F(y), F(z)) \bigr]
                \end{equation}
        \end{itemize}
    \item
        There is a partition $\cW^s$ of $Y$ into ``stable leaves''.
        The stable leaves are invariant under $F$ meaning that
        $F(W^s_y) \subset W^s_{F(y)}$ for all $y \in Y$, 
        where $W^s_y$ is the stable leaf containing $y$.
        Each $a \in \alpha$ is a union of stable leaves,
        i.e.\ $\cW_s$ is a refinement of $\alpha$.

        Let $\bY = Y / \sim$, where $y \sim z$ if $y \in W^s_z$, and
        let $\bpi \colon Y \to \bY$ be the natural projection.
        Since the stable leaves are invariant under $F$ and $\cW^s$ is a
        refinement of $\alpha$, we obtain a well defined quotient map
        $\bF \colon \bY \to \bY$ with a partition $\balpha$ of $\bY$
        and separation time $s$.

        We suppose that the probability measure $\bmu = \bpi_* \mu$ on $\bY$
        is invariant under $\bF$. Moreover, for each $a \in \balpha$, the restriction 
        $\bF \colon a \to \bY$ is a bijection (modulo $\bmu$ zero measure),
        and its inverse Jacobian $ \displaystyle \zeta_a= \frac{d\bmu}{d\bmu \circ \bF}$ satisfies
        \[
            \bigl|\log \zeta_a(y)-\log \zeta_a(z)\bigr|
            \leq K\gamma^{s(y,z)}
            \quad \text{for all } y,z \in a
            .
        \]
        In other words, the quotient map $\bF$ is full branch Gibbs-Markov.
\end{itemize}

\subsubsection{Markov shift towers}
\label{sec:MSTd}

A closely related class of nonuniformly hyperbolic maps is what we call a \emph{Markov shift tower}.
The main difference is that the induced map, denoted by $f_X$ below, is a Bernoulli shift, both topologically
and measure-theoretically.

We say that $f \colon \Delta \to \Delta$ is a \emph{Markov shift tower} if
it has the following structure.
\begin{itemize}
    \item We are given a finite or countable probability space $(\cA, \PP_\cA)$,
        an integrable function $h_\cA \colon \cA \to \{1,2,\ldots\}$ and
        a constant $\xi \in (0, 1)$. These define the rest of the construction.
    \item
        Let \((X, \PP_X) = (\cA^\Z, \PP_\cA^\Z)\) be the product probability space and let
        \(f_X \colon X \to X\) be the left shift
        \[
            f_X(\ldots, a_{-1}, a_0, a_1, \ldots) = (\ldots, a_0, a_1, a_2, \ldots)
            .
        \]
        Define \(h \colon X \to \{1,2,\ldots\}\) by \(h(\ldots, a_{-1}, a_0, a_1, \ldots) = h_\cA(a_0)\).
    \item The map $f \colon \Delta \to \Delta$ is a suspension over $f_X \colon X \to X$ with
        a roof function $h$, i.e.\
        \begin{equation}
            \begin{aligned}
                \label{eq:susp}
                \Delta & = \{(x,\ell) \in X \times \Z : 0 \leq \ell < h(x)\} \\
                f(x,\ell) & = \begin{cases} (x, \ell + 1), & \ell < h(x) - 1 \\
                (f_X (x), 0), & \ell = h(x) - 1 \end{cases}
                .
            \end{aligned}
        \end{equation}
    \item Let \(\PP\) be the probability measure on $\Delta$ defined by 
        \[
            \PP(A \times \{\ell\}) = \biggl( \int h \, d \PP_X \biggr)^{-1} \PP_X(A)
            \quad \text{for all} \quad
            \ell \geq 0
            \quad \text{and} \quad
            A \subset \{x \in X : h(x) \geq \ell + 1\}
            .
        \]
        Since \(\PP_X\) is \(f_X\)-invariant, note that \(\PP\) is \(f\)-invariant.
    \item Define a distance \(d\) on \(X\) by \(d(x,y) = \xi^{s(x,y)}\),
        where \(s \colon X \times X \to \{0,1,\ldots\}\) is the separation time,
        \[
            s((\ldots, a_{-1}, a_0, a_1, \ldots), (\ldots, b_{-1}, b_0, b_1, \ldots) )
            = \inf \{ j \geq 0 : a_j \neq b_j \text{ or } a_{-j} \neq b_{-j} \}
            .
        \]
        Let \(d\) also denote the related distance on $\Delta$:
        \begin{equation}
            \label{eq:cdis}
            d((x,k),(y,j)) = \begin{cases}
                1, & k \neq j \\
                d(x,y), & k=j
            \end{cases}.
        \end{equation}

\end{itemize}

\begin{rmk}
    As a part of the definition of a Markov shift tower, $f \colon \Delta \to \Delta$ is Markov
    both topologically (as a suspension over a Bernoulli shift, it has a natural Markov partition)
    and probabilistically (the invariant measure $\PP$ is Markov, it can be described by transition
    probabilities). Also $\Delta$ is a metric space with metric $d$.
\end{rmk}

\begin{rmk}
    \label{rmk:NE}
    We used similar Markov shift towers in our previous works~\cite{CDKM18,CDKM20SD,K17r}.
    The main difference is that they were \emph{one-sided:} the base map was the left shift
    $f_X \colon X \to X$ on $X = \cA^\N$ instead of $X = \cA^\Z$ here.
    A two-sided Markov shift tower is the natural extension~\cite{R64,S23}
    of the corresponding one-side tower.
\end{rmk}

\begin{rmk}
    An important characteristic of both nonuniformly hyperbolic maps
    and Markov shift towers is tail of the return times,
    i.e.\ the asymptotics of $\mu(\tau \geq n)$ and $\PP_X(h \geq n)$
    respectively.
\end{rmk}

\subsubsection{Main result}

\begin{thm}
    \label{thm:NUH}
    Suppose that $T \colon M \to M$ is a nonuniformly hyperbolic map as in Section~\ref{sec:NUHd}. Then there exists
    a Markov shift tower $f \colon \Delta \to \Delta$ and a map
    $\pi \colon \Delta \to M$, defined $\PP$-almost surely on $\Delta$, such that:
    \begin{itemize}
        \item $\pi$ is a semiconjugacy, i.e.\ $\pi \circ f = T \circ \pi$,
        \item $\pi$ is measure preserving, i.e.\ $\pi_* \PP = m$,
        \item $\pi$ is Lipschitz.
    \end{itemize}
    Moreover, the return time tails of $f$ are closely related to those of $T$:
    \begin{itemize}
        \item
            if $\mu(\tau \geq n) = O(n^{-\beta} L(n) )$ with $\beta > 0$ and $L$ a slowly varying function at infinity,
            then $\PP_\cA (h_\cA \geq n) = O(n^{-\beta} L(n) )$;
        \item
            if $\int \tau^\beta \, d\mu < \infty$ with $\beta > 0$,
            then $\int h_\cA^\beta \, d\PP_\cA < \infty$;
        \item
            if $\int e^{\beta \tau^\delta} \, d\mu < \infty$
            with $\beta > 0$ and $\delta \in (0,1]$, then
            $\int e^{\beta' h_\cA^\delta} \, d\PP_\cA < \infty$
            for some $0 < \beta' \leq \beta$.
    \end{itemize}
\end{thm}

\subsection{Proof of Theorem~\ref{thm:NUH}}
    
Our strategy is to reduce the proof to an application of ~\cite[Theorem~3.4]{K17r}. There, a one-sided Markov shift tower is constructed
for a nonuniformly expanding map. Then $T$ can be considered via the natural extension of such a map,
and the corresponding natural extension of the one-sided Markov shift in~\cite{K17r} is a two-sided
Markov shift, as required for Theorem~\ref{thm:NUH}.

Let $\sigma \colon \alpha^\Z \to \alpha^\Z$ be the left shift.
We supply $\alpha^\Z$ with the product topology and the Borel sigma-algebra.
For $x=(x_k)_{k \in\Z}$ and $y=(y_k)_{k \in\Z}$ in $\alpha^\Z$,
define the separation time and distance by
\begin{align*}
    s(x,y) & = \inf \{k \geq 0 : x_k \neq y_k \text{ or } x_{-k} \neq y_{-k} \}
    , \\
    d(x,y) & = \gamma^{s(x,y)}
    .
\end{align*}
Here $\gamma$ is the constant in~\eqref{eq:NUH:aa}.


\subsubsection{Tower over full shift}

Our first step is to model $T \colon M \to M$ as a suspension tower over $\sigma \colon \alpha^\Z \to \alpha^\Z$.
It is natural that this can be done, yet this requires some care.

\begin{prop}
    \label{prop:Fshift}
    There exists a probability measure $\mu_\alpha$ on $\alpha^\Z$ and a $\mu_\alpha$-almost everywhere defined
    map $\chi_\alpha \colon \alpha^\Z \to Y$ such that on the domain of $\chi_\alpha$:
    \begin{itemize}
        \item $\chi_\alpha$ is Lipschitz: $d_0(\chi_\alpha(x), \chi_\alpha(y)) \leq 2 K d(x,y)$,
        \item $(\chi_\alpha)_* \mu_\alpha = \mu$,
        \item $F \circ \chi_\alpha = \chi_\alpha \circ \sigma$,
        \item $\chi_\alpha(\ldots, a_{-1}, a_0, a_1, \ldots) \in a_0$.
    \end{itemize}
\end{prop}

\begin{proof}
    In order to construct the natural extension~\cite{R64,S23} of $F \colon Y \to Y$,
    we assume that $F(Y) = Y$. This comes without loss of generality: otherwise
    instead of $Y$  we work with $\cap_{n \geq 0} F^{-n}(Y)$. This is a full measure
    subset of $Y$ on which $F$ is surjective, and the proof goes through with straightforward
    and minor changes.

    Let $\brF \colon \brY \to \brY$ be the natural extension~\cite{R64,S23} of $F \colon Y \to Y$:
    \[
        \brY = \bigl\{
            (\ldots, y_{-1}, y_0, y_1, \ldots) \in Y^\Z : y_{n + 1} = F(y_n)
        \bigr\}
        ,
    \]
    and $\brF$ is the left shift. The topology on $\brY$ is generated by open
    cylinders of the type $\{y_i \in E\}$ with open $E \subset Y$ (in the induced topology on $Y$
    as  a subspace of $M$), and we consider $\brY$ with the Borel sigma-algebra.
    Let $\brpi_\ell \colon \brY \to Y$, $(\ldots, y_{-1}, y_0, y_1, \ldots) \mapsto y_\ell$
    be the natural projections, set $\brpi = \brpi_0$, and
    let $\brmu$ be the unique $\brF$-invariant probability measure on $\brY$ such that
    $\brpi_* \brmu = \mu$.

    Define $\iota \colon \brY \to \alpha^\Z$ by
    $(\ldots, y_{-1}, y_0, y_1, \ldots) \mapsto (\ldots, a_{-1}, a_0, a_1, \ldots)$
    where $y_n \in a_n$. Then $\iota$ is a measurable and injective map.
    Set $\mu_\alpha = \iota_* \brmu$.
    Using~\eqref{eq:NUH:aa}, for $0 \leq |\ell| \leq n$,
    \begin{align*}
            \diam \bigl( (\brpi_\ell \circ \iota^{-1})([a_{-n}, \ldots, a_n]) \bigr)
            & = \diam  \bigl( F^{n + \ell} ( \{ y \in Y : F^k(y) \in a_{-n + k} \text{ for all } 0 \leq k \leq 2n  \} ) \bigr)
            \\
            & \leq 2 K \gamma^{n - |\ell|}
            = 2 K \gamma^{-|\ell|} \diam ([a_{-n}, \ldots, a_n])
            ,
    \end{align*}
    where $[a_{-n}, \ldots, a_n] \subset \alpha^\Z$ is a cylinder.
    Hence $\brpi_\ell \circ \iota^{-1}$ is Lipschitz with Lipschitz constant $2 K \gamma^{-|\ell|}$.
    Since $\iota^{-1}(x) = ( (\brpi_\ell \circ \iota^{-1}) (x) )_{\ell \in \Z}$
    and each $\brpi_\ell \circ \iota^{-1}$ is continuous, $\iota^{-1} \colon \iota(\brY) \to \brY$ is continuous.
    Recall that, as a part of the definition of $T$ as a nonuniformly hyperbolic map we assumed that
    $\iota(\brY)$ is measurable in $\alpha^\Z$. Then $\iota^{-1}$ is measurable.

    Observe that $\iota^{-1}$ is a conjugacy between $\brF$ and $\sigma$. Set $\chi_\alpha = \brpi \circ \iota^{-1}$.
    \[
        \begin{tikzcd}[row sep=2.5em]
            \alpha^\Z \arrow[r,"\sigma"] \arrow[d,"\iota^{-1}"] & \alpha^\Z \arrow[d,"\iota^{-1}"] & &
            \mu_\alpha \arrow[d,mapsto,"(\iota^{-1})_*"]
            \\
            \brY \arrow[r,"\brF"] \arrow[d,"\brpi"] & \brY \arrow[d,"\brpi"] & &
            \brmu \arrow[d,mapsto,"(\brpi)_*"]
            \\
            Y \arrow[r,"F"] & Y & &
            \mu
        \end{tikzcd}
    \]
    Then $\chi_\alpha$ is a measure preserving semiconjugacy between $\sigma$ and $F$
    with the required properties.
\end{proof}

In the setup of Proposition~\ref{prop:Fshift},
consider the suspension map over $\sigma$ with roof function $\tau_\alpha = \tau \circ \chi_\alpha$, i.e.\ the space
\[
    M_\alpha
    = \{(x, \ell) : x \in \alpha^\Z, \; 0 \leq \ell < \tau_\alpha(x) \}
\]
with the transformation
\[
    T_\alpha \colon (x, \ell) \mapsto
    \begin{cases}
        (x, \ell + 1), & \ell < \tau_\alpha(x) - 1, \\
        (\sigma(x), 0), & \text{else.}
    \end{cases}
\]
As in \eqref{eq:cdis}, we supply $M_\alpha$ with the metric
\[
    d_\alpha((x, k), (y, \ell)) =
    \begin{cases}
        1, & k \neq \ell , \\
        d(x,y), & k = \ell .
    \end{cases}
\]
Let $m_\alpha = \mu_\alpha \times \text{counting} / \text{normalization}$
be the probability measure on $M_\alpha$, and let
$\pi_\alpha \colon M_\alpha \to M$,
\[
    \pi_\alpha \colon (x, \ell) \mapsto T^\ell(\chi_\alpha(x))
    .
\]

\begin{rmk}
    \label{rmk:Talpha}
    Observe that:
    \begin{itemize}
        \item $T_\alpha \colon M_\alpha \to M_\alpha$
            is a nonuniformly hyperbolic map with the same tails of return times as $T$,
            namely $\mu_\alpha(\tau_\alpha \geq n) = \mu(\tau \geq n)$ for all $n$.
        \item $\pi_\alpha$ is defined $m_\alpha$-almost surely, but possibly not on the whole $M_\alpha$.
        \item $\pi_\alpha \colon M_\alpha \to M$ is a measure preserving semiconjugacy between $T_\alpha$ and $T$
            (modulo zero measure).
        \item $\pi_\alpha \colon M_\alpha \to M$ is Lipschitz (on its domain). Indeed,
            \[
                d_0(\pi_\alpha(x, \ell), \pi_\alpha(y, k)) \leq C d((x,\ell), (y, k))
            \]
            holds:
            \begin{itemize}
                \item with $C = \diam M$ when $\ell \neq k$ by construction,
                \item with $C = 2K$ when $\ell = k = 0$ by Proposition~\ref{prop:Fshift},
                \item with $C = 2 K^2 ( 1 + \gamma^{-1})$ when $\ell = k > 0$. Indeed, by \eqref{eq:NUH:ab}, 
                \[
                d_0(\pi_\alpha(x, \ell), \pi_\alpha(y, \ell))  \leq K \big [ d_0 ( \chi_\alpha(x), \chi_\alpha(y))  + d_0 ( F \circ \chi_\alpha(x),   F \circ \chi_\alpha(y))  \big ] \, . \]
                Now, by Proposition~\ref{prop:Fshift},  $F \circ \chi_\alpha =  \chi_\alpha \circ \sigma$ and $\chi_\alpha$ is Lipschitz. Then 
                \[
                d_0(\pi_\alpha(x, \ell), \pi_\alpha(y, \ell))  \leq 2K^2 \big [  d ( x,y)  + d ( \sigma(x),  \sigma(y))  \big ] \, .
                \] 
                To conclude, we use the fact that since $s(x,y) \leq 1 +  s(\sigma(x),\sigma(y)) $,
                $d ( \sigma(x),  \sigma(y)) \leq \gamma^{-1} d ( x,y)$. 
            \end{itemize}
    \end{itemize}
\end{rmk}

\subsubsection{Proof of Theorem~\ref{thm:NUH} for \texorpdfstring{$T_\alpha$}{T alpha}}
\label{sec:222}

We prove Theorem~\ref{thm:NUH} for $T_\alpha$ instead of $T$,
and Remark~\ref{rmk:Talpha} guarantees that the result carries over to $T$.
We construct a number of intermediate spaces.

Let $\cA$ denote the set of all finite words in the alphabet $\alpha$,
not including the empty word.
For $w = a_0 \cdots a_{n-1} \in \cA$ let 
$h_\cA(w) = \tau(a_0) + \cdots + \tau(a_{n-1})$.

Define $\pi_\cA \colon \cA^\Z \to \alpha^\Z$ by
$\pi_\cA(\ldots, w_{-1}, w_0, w_1, \ldots ) = \cdots w_{-1} w_0 w_1 \cdots$,
which is a concatenation with the first letter of $w_0$ at index $0$.
Similarly define $\bpi_\cA \colon \cA^\N \to \alpha^{\N}$ by
$\bpi_\cA(w_0, w_1, \ldots ) = w_0 w_1 \cdots$.

Let $\pi_\alpha^+ \colon \alpha^\Z \to \alpha^\N$ be the projection on nonnegative coordinates.
Using that $\tau_\alpha$ depends only on the ``future'' coordinates,
define $\btau_\alpha \colon \alpha^{\N} \to \N_0$ so that
$\tau(x) = \btau_\alpha(\pi_\alpha^+(x))$.

Let $\bM_\alpha = \{ (x, \ell) \in \alpha^{\N} \times \Z : 0 \leq \ell < \btau_\alpha(x) \}$, and let
\[
    \bT_\alpha \colon (x, \ell) \mapsto
    \begin{cases}
        (x, \ell + 1)   & \ell < \btau_\alpha(y) - 1 , \\
        (\sigma(x), 0)     & \ell = \btau_\alpha(y) - 1 . 
    \end{cases}
\]
Then $\bT_\alpha \colon \bM_\alpha \to \bM_\alpha$ is a nonuniformly \emph{expanding} system.
It preserves the probability measure $\bm_\alpha = \bmu_\alpha \times \text{counting} / \text{normalization}$,
where $\bmu_\alpha = (\pi_\alpha^+)_* \mu_\alpha$ is the projection of $\mu_\alpha$ on nonnegative coordinates.

Let $\psi \colon M_\alpha \to \bM_\alpha$ be the natural projection,
\[
    \psi \colon ( x, \ell ) \mapsto ( \pi_\alpha^+(x), \ell \bigr)
    .
\]
\begin{rmk}
    \label{rmk:ku}
    $T_\alpha$ is the natural extension of $\bT_\alpha$. In particular, $\psi$ is a semiconjugacy, and
    $\bm_\alpha$ is the unique $\bT_\alpha$-invariant probability measure so that $\psi_* m_\alpha = \bm_\alpha$.
\end{rmk}

To use the same notations as in~\cite{K17r}, let $\xi = \gamma$.
By~\cite[Theorem~3.4, Section~4.1]{K17r}, there exists a probability measure $\PP_\cA$ on $\cA$,
such that the \emph{one-sided} Markov shift tower (see Remark~\ref{rmk:NE}) $\bff \colon \bDelta \to \bDelta$,
defined by $\bigl\{ (\cA, \PP_\cA)$, $h_\cA$, $\xi \bigr\}$, is an extension of $\bT_\alpha$
with the measure preserving and Lipschitz semiconjugacy
$\bpi_\alpha \colon \bDelta \to \bM_\alpha$,
\[
    \bpi_\alpha \colon (x, \ell) \mapsto \bT_\alpha^\ell (\bpi_\cA(x),0)
    .
\]
Next let $\bPP$ denote be the $\bff$-invariant probability measure on $\bDelta$.

Now let $f \colon \Delta \to \Delta$ be the \emph{two-sided} (i.e.\ as in Section~\ref{sec:NUH}) Markov shift tower
defined by the same objects $\bigl\{ (\cA, \PP_\cA)$, $h_\cA$, $\xi \bigr\}$,
with the invariant probability measure $\PP$.

Let $\bpi_\Delta \colon \Delta \to \bDelta$ be the natural projection,
\[
    \bpi_\Delta \colon (x, \ell) \mapsto (\bpi^+_\cA(x), \ell)
    ,
\]
where $\bpi^+_\cA :  \cA^\Z \to \alpha^\N$ is defined by  $\bpi^+_\cA  (\ldots, w_{-1}, w_0, w_1, \ldots ) =w_0 w_1 w_2 \cdots$. Observe that $\bpi_\Delta$ is a measure preserving and Lipschitz semiconjugacy between
$f$ and $\bff$.

\begin{rmk}
    $\bpi_\alpha \circ \bpi_\Delta \colon \Delta \to \bM_\alpha$ is a measure preserving semiconjugacy,
    in particular
    \[
        (\bpi_\alpha \circ \bpi_\Delta)_* \PP
        = \bm_\alpha
        .
    \]
\end{rmk}

Define $\pi_\alpha \colon \Delta \to M_\alpha$,
\[
    \pi_\alpha \colon (x, \ell) \mapsto T_\alpha^\ell (\pi_\cA(x), 0)
    .
\]
Observe that $\pi_\alpha$ is a measure preserving semiconjugacy too.

Now, $f $ ($ \Delta \to \Delta$) preserving the probability measure $\PP$,
with the semiconjacy $\pi_\alpha$ to $T_\alpha $ ($ M_\alpha \to M_\alpha$),
is the Markov shift tower we are after. It remains to verify that it has the required properties:
that $\pi_\alpha$ is measure preserving and Lipschitz, and to bound the return time tails. 

\begin{prop}
    $\pi_\alpha$ is measure preserving:
    $(\pi_\alpha)_* \PP = m_\alpha$.
\end{prop}

\begin{proof}
    We have constructed four dynamical systems:
    $f \colon \Delta \to \Delta$ with invariant measure $\PP$,
    $\bff \colon \bDelta \to \bDelta$ with invariant measure $\bPP$,
    $T_\alpha \colon M_\alpha \to M_\alpha$ with invariant measure $m_\alpha$, and
    $\bT_\alpha \colon \bM_\alpha \to \bM_\alpha$ with invariant measure $\bm_\alpha$.
    They are connected by semiconjugacies $\pi_\alpha \colon \Delta \to M_\alpha$,
    $\bpi_\alpha \colon \bDelta \to \bM_\alpha$, $\bpi_\Delta \colon \Delta \to \bDelta$
    and $\psi \colon M_\alpha \to M_\alpha$, where $\psi$ is the natural projection.

    The left diagram commutes, and we have to justify the dashed arrow in the right diagram:
    \begin{equation}
        \label{eq:ccomm}
        \begin{tikzcd}[row sep=2.5em,column sep=1.5em]
            \Delta \arrow[dd,"\pi_\alpha"'] \arrow[rd,"\bpi_\Delta"] & & &
            & &
            \PP \arrow[dd,mapsto,dashed,"(\pi_\alpha)_*"'] \arrow[rd,mapsto,"(\bpi_\Delta)_*"] & &
            \\
            & \bDelta \arrow[rd,"\bpi_\alpha"] & &
            & & & \bPP \arrow[rd,mapsto,"(\bpi_\alpha)_*"] &
            \\
            M_\alpha \arrow[rr,"\psi"] & & \bM_\alpha & 
            & &
            m_\alpha \arrow[rr,mapsto,"\psi_*"] & & \bm_\alpha
        \end{tikzcd}
    \end{equation}
    By Remark~\ref{rmk:ku}, $m_\alpha$ is the \emph{unique} probability measure
    on $M_\alpha$ so that $\psi_* m_\alpha = \bm_\alpha$.
    On the other hand, $\psi \circ \pi_\alpha = \bpi_\Delta \circ \bpi_\alpha$, and therefore
    $\psi_* \bigl( (\pi_\alpha)_* \PP \bigr) = \bm_\alpha$, so $(\pi_\alpha)_* \PP = m_\alpha$
    as required.
\end{proof}

By the same argument as in~\cite[Proposition~4.18]{K17r}, $\pi_\alpha$ is Lipschitz.

The return time tails on $\Delta$ and $\bDelta$ are equal by construction,
and~\cite[Theorem~3.4]{K17r} proves that on $\bDelta$ they are as required for
Theorem~\ref{thm:NUH} except for the more general than in~\cite{K17r} polynomial case
$\mu_\alpha(\tau_\alpha \geq n) = O(n^{-\beta} L(n))$
with $\beta > 0$ and $L$ a slowly varying function at infinity.
In turn, the relation between the tails in~\cite{K17r} is taken from~\cite[Section 4]{K18e}.
To extend the results as we require, it is sufficient to notice that,
by taking into account the properties of slowly varying functions (see for instance~\cite[Chapter~1]{BGT87}), the following version
of~\cite[Proposition~4.4]{K18e} holds:

\begin{prop}[in notation of~\cite{K18e}]
    \label{prop:pollog}
    Suppose that there exist $C_\tau > 0$, $\beta > 0$, $L$ a slowly varying function at infinity and $\ell_0 \geq 1$, such that
    $m(\tau \geq \ell) \leq C_\tau \ell^{-\beta} L( \ell)$ for all $\ell \geq \ell_0$.
    Then $\PP(t \geq \ell) \leq C \ell^{-\beta}L( \ell)$.
\end{prop}

The proof of Theorem~\ref{thm:NUH} for $T_\alpha$ is complete, and following Remark~\ref{rmk:Talpha} we recover the full result.

\section{ASIP for H\"older observables of nonuniformly hyperbolic maps}
\label{sec:ASIPNUH}

Theorem~\ref{thm:examples} is an application of the following general result, which is the goal of this section:

\begin{thm}
    \label{thm:ASIPNUH}
    Let $T \colon M \to M$ be an ergodic, measure-preserving transformation defined on a bounded metric
space $(M, d)$ with Borel probability measure $m$. Let  $\varphi \colon M \to \R$ be H\"older continuous and centered, i.e.\ $\int \varphi \, dm = 0$.
    Let $S_n( \varphi)  = \sum_{k =0}^{ n-1} \varphi \circ T^k$. Suppose that $T$ is nonuniformly hyperbolic (in the sense of Section~\ref{sec:NUHd}) with uniformly hyperbolic induced map $F$, return time $\tau$ and  $F$-invariant measure 
    $\mu$ associated with a subset $Y$ of $M$.  Assume that $\int \tau^2 \, d \mu < \infty$. Then  the limit 
    $c^2=   \lim_{n \rightarrow \infty}  n^{-1}\int |S_n( \varphi)|^2 d \mu$ exists. In addition, 
    \begin{enumerate}[label=(\alph*)]
        \item\label{thm:ASIP:flower:NUH}
            If $ \mu ( \tau \geq n) = O(n^{-\beta} (\log n)^\gamma)$, with $\beta > 2$ and $\gamma \in {\mathbb R}$,   
            then for each $\eps > 0$ the process $S_n$ satisfies the ASIP with variance $c^2$ and rate
            $o(n^{1/\beta} (\log n)^{( \gamma + 1)/\beta+ \eps})$.
        \item If $\int \tau^\beta \, d \mu < \infty$ with $\beta > 2$,
            then $S_n$ satisfies the ASIP with variance $c^2$ and rate
            $o(n^{1/\beta})$.
        \item If $\int e^{\beta \tau^\delta} \, d \mu  < \infty$ with $\beta > 0$ and $\delta \in (0,1]$,
            then $S_n$ satisfies the ASIP with variance $c^2$ and rate
            $O(\log n)^{1 + 1/\delta}$.
    \end{enumerate}
\end{thm}

\begin{rmk}
    The dependence of the ASIP on $\eps$ in item~\ref{thm:ASIP:flower:NUH} should be understood as follows:
    for every $\eps > 0$, there exists a Brownian motion $(W_t)_{t >0}$ with variance $c^2$ such that
    $S_n = W_n + o(n^{1/\beta} (\log n)^{( \gamma + 1)/\beta+ \eps})$ almost surely.
\end{rmk}

For the flower billiard map and the two falling balls map   defined  in~\eqref{eq:MAPS}, it has been proved in  \cite{CZ05a} that they are examples of nonuniformly hyperbolic maps with tails of the return times of the form $O ( n^{-3} ( \log n)^3 )$. On another hand, for the  intermittent Baker's map as described in \eqref{eq:LSV2}, it has been proved in \cite{MV16}  that the associated return times have tails  $\sim n^{-\beta}$ with $\beta =1/\alpha$.  These considerations together with Theorem \ref{thm:ASIPNUH} prove Theorem \ref{thm:examples} of the introduction. 

\begin{rmk}
    \label{rmk:BZ23}
    It is expected~\cite{BZ23} that for the maps~\eqref{eq:MAPS} the tails of the return times can be improved to $O(n^{-3})$,
    dropping the logarithmic factor, similarly to what was done for Buminovich stadia and related systems~\cite{CZ08}.
    This would improve our result.
\end{rmk}

Let $f \colon \Delta \to \Delta$ be a Markov shift tower constructed for $T$ as in Theorem~\ref{thm:NUH},
built with a probability space $(\cA, \PP_\cA)$, roof function $h_\cA$
and the metric constant $\xi \in (0,1)$. Recall the associated notations,
in particular that $(\Delta, d)$ is a metric space and that $f$ preserves the probability measure $\PP$.

By Theorem~\ref{thm:NUH}, the process $( \varphi \circ T^k)_k$ defined on $(M,m)$ and the process
$( v \circ f^k)_k$, where $v = \varphi \circ \pi$,  defined on $(\Delta,\PP)$ have the same law.
Moreover, if $ \varphi$ is H\"older continuous then so is $v$ since $\pi$ is Lipschitz.
Hence Theorem~\ref{thm:ASIPNUH} will follow from its equivalent result stated for Markov shift towers
as given in Proposition~\ref{thm:ASIPProp} below.

\subsection{ASIP for Markov shift towers}
\label{sec:ASIP}

\begin{prop}
    \label{thm:ASIPProp}
    Suppose that $v \colon \Delta \to \R$ is H\"older continuous and centered, i.e.\ $\int v \, d\PP = 0$.
    Let $S_n = \sum_{k < n} v \circ f^k$. Assume that $\int h_\cA^2 \, d\PP_\cA < \infty$. Then $c^2=   \lim_{n \rightarrow \infty} n^{-1}\int |S_n( \varphi)|^2 d \mu$ exists. In addition, 
    \begin{enumerate}[label=(\alph*)]
        \item\label{thm:ASIP:flower}
            If $\PP_\cA (h_\cA \geq n) =  O(n^{-\beta} (\log n)^\gamma)$, with $\beta > 2$ and $\gamma \in {\mathbb R}$, 
            then $S_n$ satisfies the ASIP with variance $c^2$ and rate
            $o(n^{1/\beta} (\log n)^{ (\gamma + 1)/\beta+ \eps})$ for each $\eps > 0$.
        \item If $\int h_\cA^\beta \, d\PP_\cA < \infty$ with $\beta > 2$,
            then $S_n$ satisfies the ASIP with variance $c^2$ and rate
            $o(n^{1/\beta})$.
        \item If $\int e^{\beta h_\cA^\delta} \, d\PP_\cA < \infty$ with $\beta > 0$ and $\delta \in (0,1]$,
            then $S_n$ satisfies the ASIP with  variance $c^2$ and rate
            $O(\log n)^{1 + 1/\delta}$.
    \end{enumerate}
\end{prop}

The rest of this section is devoted to the proof of Proposition \ref{thm:ASIPProp}.

\subsection{The associated Markov chain}

It is convenient to represent the dynamics on a Markov shift tower
as a Markov chain in the conventional sense.
For this, let $G = \{(a, \ell) \in \cA \times \Z : 0 \leq \ell < h_\cA(a)\}$
and let $\cG \subset G^\Z$ be the set of admissible symbolic trajectories:
\[
    \cG = \{ g = (g_n)_{n \in \Z} \in G^\Z : \text{ if } g_n = (a, \ell) \text{ with }
        \ell < h(a) - 1 \text{, then } g_{n+1} = (a, \ell + 1)
    \}.
\]
Let now $(g_n)_{n\geq 0}$ be a Markov chain with state space $G$ and transition\
probabilities
\begin{align*}
    \PP (g_{n+1} = (a, \ell) & \mid g_n = (a', \ell'))
    \\
    & =
    \begin{cases}
        1, & \ell = \ell' + 1 \text{ and } \ell' + 1 < h(a) \text{ and } a = a', \\
        \PP_\cA(a) , & \ell = 0 \text{ and } \ell' + 1 = h(a'), \\
        0, & \text{else}.
    \end{cases}
\end{align*}
The Markov chain has a unique invariant probability measure  $\nu$
with respect to which the Markov chain $(g_n)_{n \geq 0}$ is stationary.
By the Kolmogorov existence theorem, there exists a stationary Markov chain
indexed by $\Z$ that we still denote by $(g_n)_{n \in \Z}$ with transition
probabilities given above and which defines a probability measure $\PP_{\cG}$
on the space $\cG$. 

\smallskip

Define $\pi_G \colon \Delta \to \cG$ by $\pi_G(x) = (a_n, \ell_n)_n$
such that $f^n(x) \in ([a_n], \ell_n)$ for all $n$, where $[a_n] \subset \cA^\Z$
is the cylinder $\{ x \in \cA^\Z : x_0 = a_n \}$.
Note that $\pi_G$ is a bijection; moreover, it is 
a measure preserving conjugacy.  It follows that if we define   
\begin{equation} \label{defpsi}
\psi \colon \cG \to \R \text{ by }\psi = v \circ \pi_G^{-1}, 
\end{equation} then setting for all $k \in \Z$,
\begin{equation} \label{defprocessXk}
    X_k = \psi ( (g_{k+n})_{n\in \Z} ) , 
\end{equation} 
the stationary process $(X_k)_{k \in \Z} $ has the same law as $( v \circ f^k)_{k \in \Z} $. To give the regularity properties of $\psi$, we need to equip  
$\cG$ with a suitable metric.

\subsection{Metric on \texorpdfstring{$\cG$}{G}}

Since $\pi_G \colon \Delta \to \cG$ is a bijection and $\Delta$ is a metric space,
it is natural to define a metric on $\cG$ so that $\pi_G$ is an isometry.
%
For $g, g' \in \cG$, let
\[
    \ts^\pm (g, g')
    = \inf \{ \ell \geq 0 : g_{\pm \ell} \neq g'_{\pm \ell} \}
    .
\]
Let $G_0 = \{(a, 0) : a \in \cA \} \subset G$, 
\[
    s^{-} (g, g')
    = \# \big\{ 0 \leq \ell <  \ts^{-} (g, g') : g_{- \ell} \in G_0 \big\}
    \text{ and }  s^{+} (g, g')
    = \# \big\{ 0 < \ell \leq \ts^{+} (g, g') : g_{ \ell} \in G_0 \big\} .
\]
Let
\[
    s (g, g')
    = \min \bigl\{ s^- (g, g'), s^+ (g, g') \bigr\}
    .
\]
Then $s$ is a kind of separation time (cf.~\ the separation
time on $X$ defined in Section~\ref{sec:NUH}). 
Recall that $\xi \in (0,1)$ and define
\[
    d(g, g') = \xi^{s(g, g')}
    .
\]
This is the metric on $\cG$ which agrees with that on $\Delta$ in the sense that
$d(x,y) = d(\pi_G(x), \pi_G(y))$ for all $x,y \in \Delta$.

This allows to infer that if $v$ is H\"older with index $\eta \in ]0,1]$ then so is $\psi$ defined by \eqref{defpsi}. 
\smallskip

Next, to  prove the ASIP  for the partial sums associated with $(X_k)_{k \in \Z}$ defined by \eqref {defprocessXk} with $\psi$ bounded and H\"older continuous, it is convenient, as in \cite{CDKM18, CDKM20SD}, to represent the Markov chain with the help of its innovations and to introduce a particular meeting time.

\subsection{Innovations and meeting time of the  underlying Markov chain} 
\label{sec:inn}

Extending the approach in~\cite[Section~3]{CDKM18},
without changing the distribution of $(g_n)_{n \in \Z}$, we assume that
there is a sequence of iid random elements $(\eps_n)_{n \in \Z}$ sampled from $(\cA, \PP_\cA)$
such that each $\eps_{n+1}$ is independent from $(\eps_\ell, g_\ell)_{\ell \leq n}$,
and $g_{n+1}=U(g_n,\eps_{n+1})$, where
\[
    U((a, \ell), \eps)
    = \begin{cases}
        (a, \ell+1), & \ell < h(a) - 1, \\
        (\eps, 0), & \ell = h(a) - 1.
    \end{cases}
\]
We refer to $(\eps_n)_{n \in \Z}$ as \emph{innovations}.

For $g_0, g'_0 \in G$ and a sequence $(\eps_n)_{n \in \Z}$ of random elements  in $\cA$, let $T_{g_0, g'_0} ((\eps_n))$
be the \emph{meeting time} of two Markov chains
with respective initial states $g_0, g'_0$ and common innovations $(\eps_n)_{n \in \Z}$:
\[
    T_{g_0, g'_0} ((\eps_n)) = \inf \{ \ell \geq 0 : g_\ell = g'_\ell \}
    ,
\]
where $g_{n+1} = U(g_n, \eps_{n+1})$ and $g'_{n+1} = U(g'_n, \eps_{n+1})$.

\begin{prop}
    \label{prop:MTB}
    Suppose that $g'_0$ is an independent copy of $g_0$, also independent from $(g_n, \eps_n)_{n \geq 0}$. 
    Set $T = T_{g_0, g'_0} ((\eps_n))$. 
    \begin{enumerate}
        \item  Assume that there exists a sequence  $\eps(n)$ of positive
    reals tending to $0$ as $n \rightarrow \infty$, such that $\PP_\cA( h_\cA \geq n) \geq \exp (-n \eps(n))$. If in addition $\PP_\cA( h_\cA \geq n)  = O (  L(n) n^{-\beta}) $ with $\beta > 1$
            and $L(n)$ a slowly varying function at infinity, then  $\PP (T \geq n )  = O (  L(n) n^{-(\beta -1)})$.
        \item  If $\PP_\cA( h_\cA \geq n)  = O (  L(n) n^{-\beta}) $ with $\beta > 1$
            and $L(n)$ a slowly varying function at infinity, then  $ \E ( g_\beta(T) )  <\infty$
            for any $\eta >1$ where $g_\beta(x) = x^{\beta-1} (\log (1+x))^{- \eta}/ L(x)$. 
        \item  If $\E_\cA( h^\beta_\cA)  < \infty$ for some $\beta >1$, then $\E ( T^{\beta-1} ) < \infty$.
        \item  If  $\PP_\cA( h_\cA \geq n) = O (  {\rm e}^{ - c n^\delta}  )< \infty$ for some $c >0$
            and $\delta \in (0,1]$ , then  there exists $\kappa >0$ such that 
            $ \PP (T \geq n ) = O (  {\rm e}^{ - \kappa n^\delta}  )$.  
    \end{enumerate}
\end{prop}
Items 2 and 3 are proved in \cite{CDKM18} whereas Item 4 is proved in \cite{CDKM20SD}.
Item 1 of Proposition \ref{prop:MTB} follows straightforwardly from Lemma \ref{lem:MTB1} below.

\begin{lemma}\label{lem:MTB1}
    Suppose that $g'_0$ is an independent copy of $g_0$, also independent from $(g_n, \eps_n)_{n \geq 0}$. 
    Set $T = T_{g_0, g'_0} ((\eps_n))$. Suppose that there exists a sequence  $\eps(n)$ of positive
    reals tending to $0$ as $n \rightarrow \infty$, such that $\PP_\cA( h_\cA \geq n) \geq \exp (-n \eps(n))$. Then
    \[
        \PP( T \geq n) = O \bigl( \E_\cA( (h_\cA -n)_+)  \bigr)
        .
    \]
\end{lemma}

\begin{proof}
    As in Lindvall~\cite{Lind} (see also Rio~\cite[Proposition 9.6]{Rio17}),
    let $\Lambda_0$ be the class of nondecreasing functions $\psi$ from ${\mathbb N}$ to $[1, \infty[$
    such that $\log (\psi (n))/n$ decreases to $0$ as $n \rightarrow \infty$. Let
    $$
        \psi^{(0)}(k) = \sum_{i=0}^{k-1} \psi(i)
        \quad \text{and, for $k \geq 1$,} \quad
        \varphi(k) = \psi(k)-\psi(k-1)
        .
    $$
    Proceeding exactly as in the proof of~\cite[Lemma~3.1]{CDKM18} and applying the result of~\cite{Lind}, we infer that 
    $$
        \E_\cA \bigl ( \psi^{(0)}(h_\cA) \bigr) < \infty
        \Rightarrow \E(\psi(T)) < \infty
        .
    $$
    This last assertion is in fact equivalent to 
    \begin{equation}\label{eq:implies}
        \sum_{k=1}^\infty \varphi(k) \E_\cA \bigl( (h_\cA -k)_+ \bigr) < \infty
        \Rightarrow
        \sum_{k=1}^\infty \varphi(k) \PP( T \geq k ) < \infty
        .
    \end{equation}
    Now, assume that the conclusion of Lemma~\ref{lem:MTB1} is not true.
    Then there exists an increasing sequence $(n_k)_{k \geq 1}$ such that $\PP( T \geq n_k) \geq  k \E_\cA( (h_\cA -n_k)_+)$.
    Define then the function $\varphi$ as follows: $\varphi(i)=0$ if $i\notin \{n_k, k \geq 1\}$
    and $\varphi(n_k)= k^{-3/2}/\E_\cA( (h_\cA -n_k)_+) $.
    For such a $\varphi$ it is clear that the sum on left hand in~\eqref{eq:implies} is finite,
    while the selection of $n_k$ implies that the sum on right hand is $+\infty$.
    This leads to a contradiction and proves the Lemma, provided we check that the function $\psi$ defined by 
    $\psi(k) =\psi(0) + \varphi(1)+ \cdots + \varphi(k)$ belongs to $\Lambda_0$.
    Hence, it remains to check that $\log(\psi(n_k))/n_k$ tends to $0$ as 
    $k\rightarrow \infty$. Now, by definition of $\varphi$,
    $\psi(n_k)-\psi(0) \leq C\sqrt k/\E_\cA( (h_\cA -n_k)_+)$. On another hand, the assumption on $h_\cA$ implies that ${\mathbb E}_\cA( (h_\cA -n_k)_+) \geq \exp (-n_k \eps'(n_k))$ for some sequence $\eps'(n)$ of positive reals tending to $0$ as $n \rightarrow \infty$. This implies that $\log(\psi(n_k))/n_k$ tends to 0 as 
    $k\rightarrow \infty$, and the proof is complete.
\end{proof}

\subsection{Proof of Proposition \ref{thm:ASIPProp} when \texorpdfstring{$c^2 > 0$}{variance is non-zero}}
\label{sec:non-trivial}

By the previous considerations, it suffices to prove the ASIP for the partial sums associated with $(X_k)_{k \in \Z}$ defined by \eqref {defprocessXk} with $\psi$ bounded and H\"older continuous, by taking into account Proposition \ref{prop:MTB}.

Without loss of generality, one can assume that $v \colon \Delta \to \R$ is Lipschitz and that $\|v\|_{\Lip} \leq 1$,
where $\|v\|_{\Lip} = \sup_x |v(x)| + \sup_{x \neq y} |v(x) - v(y)| / d(x,y)$.
Recall that $\psi \colon \cG \to \R$, $\psi = v \circ \pi_G^{-1}$.
Then $\psi$ is also Lipschitz, in particular
\[
   | \psi(g) - \psi(g') |
    \leq \xi^{s(g,g')}
   .
\]
%
We assume also in the rest of this section that the Markov chain $(g_n)_{n \in \Z}$ is aperiodic. If it is not the case,
we modify the proof as in  \cite[Appendix~A]{CDKM18}.

The proof follows the lines of \cite[Section~4.2]{CDKM18} or \cite[Section~3]{CDKM20SD} with the same notations, but with the following changes. \cite[Proposition~3.2]{CDKM18} and~\cite[Proposition~2.3]{CDKM20SD} are replaced by Proposition \ref{Prop34} below whereas \cite[Inequality~4.10]{CDKM18} is replaced by our Proposition \ref{Ine410}, whereas \cite[Lemma~3.3]{CDKM18}  is replaced by our Lemma \ref{lma33}.  Moreover in the case where $\PP_\cA (h_\cA \geq n) =  O(n^{-\beta} (\log n)^\gamma)$, with $\beta > 2$ and $\gamma \in {\mathbb R}$, we take $m=m_{\ell} = [3^{\ell/\beta} \ell^{\kappa}]$ with $\kappa > (1+\gamma)/\beta$. In case  $\int h_\cA^\beta \, d\PP_\cA < \infty$ with $\beta > 2$,
    we select $m=m_{\ell} = [3^{\ell/\beta} ]$, and if  $\int e^{\beta h_\cA^\delta} \, d\PP_\cA < \infty$ with $\beta > 0$ and $\delta \in (0,1]$, we take 
  $m=m_{\ell} = [\kappa \ell^{1/\delta} ] $, for a suitable $\kappa$.

\medskip

To state the key propositions   \ref{Prop34} and  \ref{Ine410}, we need to introduce some notations.

Fix $k \in \Z$ and $m > 0$, and let $(\eps'_n)_n$ be a copy of $(\eps_n)_n$
independent of $(g_n, \eps_n)_n$. Define
\[
    \tg_\ell =
    \begin{cases}
     g_\ell,                        & \ell \leq k + m , \\
        U( g_{\ell - 1}, \eps'_\ell) ,&    \ell =  k + m    +1   ,   \\
        U( \tg_{\ell - 1}, \eps'_\ell), & \ell \geq  k + m +2
    \end{cases}
    \qquad \text{and} \qquad
    X_{m,k} = \E_g \psi( (\tg_{n+k})_{n \in \Z} )
    ,
\]
where $\E_g$ is the conditional expectation given $(g_n, \eps_n)$. 

Let 
\begin{equation} \label{defTmeet}
T = T_{g_0, g'_0}((\eps_n)_{n\geq 1})
\end{equation}
be  the meeting time between $(g_n)_{n \geq 0}$ and $(\tg_{n})_{n \geq 0}$ with the same innovation $(\eps_n)_{n \geq 1}$ but starting from two independent 
starting points $g_0$ and $g_0'$.  
\begin{prop}  \label{Prop34} $ $ 
\begin{enumerate} 
   \item Assume that $\E (T) < \infty$. Then, for every $r \geq 1$,
    \[
        \E | X_{m,k} - X_k |
        \ll m^{-r/2} + \PP( T \geq \lfloor m / r \rfloor)
        .
    \]
    \item Assume that there exist $\delta >0$ and $\gamma \in ]0,1]$ such that $\PP( T \geq n) = O ( {\rm e}^{- \delta n^{\gamma}} )$. Then there exists 
    $\delta' >0$ such that 
    \[
      \E | X_{m,k} - X_k | = O ( ( {\rm e}^{- \delta' m^{\gamma}} ) .
    \]

    \end{enumerate}
\end{prop}

\begin{proof}
    See~\cite[Proposition~3.2]{CDKM18} and~\cite[Proposition~2.3]{CDKM20SD}. 
\end{proof}

Now define
\begin{equation}
    \label{eq:tXmk}
    \tX_{m,k} = \E ( X_{m,k} \mid \eps_{k-m}, \ldots, \eps_{k+m} )
    .
\end{equation}

\begin{prop}
    \label{Ine410}
    Let $\theta^{-}_u = \sum_{k=0}^{u-1} {\bf 1}_{\{ g_{-k} \in G_0 \}}$
    and $\theta^{+}_u = \sum_{k=0}^{u-1} {\bf 1}_{\{ g_{k} \in G_0 \}}$.  Then
    \[
        \E | \tX_{m,k} - X_{m,k} |
        \ll   \| \psi \|_\infty \PP (T > m/2)
        + \E \xi^{  \min (\theta^{-}_{[m/2]}, \theta^{+}_{m}  ) } , 
    \]
    where $T$  is defined by \eqref{defTmeet}. Consequently:
    \begin{enumerate} 
   \item If $\E (T) < \infty$, then, for every $r \geq 1$,
    \[
        \E | \tX_{m,k} - X_k |
        \ll m^{-r/2} + \PP( T \geq \lfloor m / r \rfloor)
        .
    \]
    \item If there exist $\delta >0$ and $\gamma \in ]0,1]$ such that $\PP( T \geq n) = O ( {\rm e}^{- \delta n^{\gamma}} )$, then there exists 
    $\delta' >0$ such that 
    \[
      \E | \tX_{m,k} - X_k | = O ( {\rm e}^{- \delta' m^{\gamma}} ) .
    \]

    \end{enumerate}
\end{prop}

\begin{proof}
    Since $X_{m,k}$ is determined by $(g_n)_{n \leq k + m}$, we write
    $X_{m,k} = h_m((g_n)_{n \leq k + m})$ with some function $h_m$.
    Let $(g'_n, \eps'_n)_{n \in \Z}$ be an independent copy of $(g_n, \eps_n)_{n \in \Z}$.
    Note that $g_n' = U (g'_{n-1}, \eps'_n)$.
    Define $( \hg_{\ell}) $ by 
    \[
        \hg'_\ell = 
        \begin{cases}
            g'_\ell & \ell  <   k - m , \\
            U ( g'_{k-m -1}, \eps_{k-m} ) ,       & \ell  =   k - m , \\
            U (\hg'_{\ell-1 }, \eps_{\ell} ) ,    & \ell \leq k - m + 1 .
        \end{cases}
    \]
    Let $\cB_{k,m}$ be the sigma-algebra generated by
    $\bigl( (g_{\ell})_{\ell < k-m} , \eps_{k-m}, \ldots, \eps_{k+m} \bigr)$. 
    By the properties of conditional expectation, we have
    \[
        X_{m,k} = \E \bigl( h_m ((g_n)_{n \leq k + m}) \mid \cB_{k,m} \bigr)
    \]
    and
    \[
        \tX_{m,k} = \E \bigl( h_m  \bigl( (g'_u)_{u < k - m} , (\hg'_u)_{k - m \leq u \leq k+m}  \bigr) \mid \cB_{k,m} \bigr)
        .
    \]
    Set 
    \[
        \hT = T_{g_{k-m-1}, \hg'_{k-m-1} } ( ( \eps_{\ell} )_{\ell \geq k - m}) 
    \]
    This is the meeting time of the chains $(g_n)_{n \geq k - m - 1}$ and $(\hg'_n)_{n \geq k - m-1}$
    with innovations $( \eps_{\ell} )_{\ell \leq k - m}$ and independent starting points $g_{k-m-1}$ and $\hg'_{k-m-1}$. It has the same law as $T$ defined in \eqref{defTmeet}. 
    Note that 
    \begin{align*}
        I
       &  = \bigl| h_m ((g_n)_{n \leq k + m}) - h_m  \bigl((g'_u)_{u \leq k - m-1} , (\hg'_u)_{k - m \leq u \leq k + m}  \bigr) \bigr| \\
       &  \leq 2 \Vert h_m \Vert_{\infty} {\mathbf 1}_{\hT >m/2} 
        + \sum_{\ell =0}^{[m/2]}  {\mathbf 1}_{\hT= \ell}   \Bigl| h_m ((g_n)_{n \leq k + m})  \\
  &  \quad \quad  \quad - h_m  \bigl( (g'_u)_{u \leq k - m-1} , (\hg'_u)_{k - m \leq u < k - m + \ell},
        g_{k - m + \ell}, \dots, g_k, \dots, g_{k+m}  \bigr) \Bigr| 
        .
    \end{align*}
    But, from the definition of the separation distance and using $\|v\|_{\Lip} \leq 1$, for any  $\ell$ in $\{0, \dots, [m/2] \}$, 
    \begin{multline*}
        \bigl| h_m ((g_n)_{n \leq k + m})  
        - h_m  \bigl((g'_u)_{u \leq k - m-1} , (\hg'_u)_{k - m \leq u < k-m + \ell},
        g_{k - m + \ell}, \dots, g_k, \dots, g_{k+m}  \bigr) \bigr|
        \\
        \leq  \max \big (  \xi^{\# \{ k - m + \ell \leq i \leq k : g_i \in G_0 \} }  ,  \xi^{\# \{ k \leq i \leq k +m : g_i \in G_0 \} }  \big ) .
    \end{multline*}
    Taking the expectation, we get 
    \begin{multline*}
        \E |X_{m,k} -  \tX_{m,k} |
        \leq \E ( I  )
        \\
        \leq 2 \E \bigl( \Vert h_m \Vert_{\infty} {\mathbf 1}_{\hT >m/2} \bigr)
        + \E \max \bigl(  \xi^{\# \{ k - [m/2] \leq i \leq k : g_i \in G_0 \} }  ,  \xi^{\# \{ k \leq i \leq k +m : g_i \in G_0 \} }  \bigr)
        .
    \end{multline*}
    The first part of the proposition  follows by stationarity and the fact that $ \Vert  h_m \Vert_{\infty} \leq \|\psi \|_\infty $.  To end the proof, we used the same arguments as in the proofs of \cite[Proposition~3.2]{CDKM18} and~\cite[Proposition~2.3]{CDKM20SD}. 
\end{proof}

\begin{lemma} \label{lma33} $ $ 
\begin{enumerate} 
   \item If $\E (T) < \infty$, for every $\alpha \geq 1$ and every $k \geq 1$, 
\[
 | {\rm Cov} (X_0,X_k) |  \ll   k^{-\alpha/2} + \PP( T \geq \lfloor k / (4 \alpha)  \rfloor) 
 .
\]
\item If there exist $\delta >0$ and $\gamma \in ]0,1]$ such that $\PP( T \geq n) = O ( {\rm e}^{- \delta n^{\gamma}} )$, then for any $k \geq 0$, there exists 
    $\delta' >0$ such that 
    \[
 | {\rm Cov} (X_0,X_k) |  = O ( ( {\rm e}^{- \delta' k^{\gamma}} ) .
    \]
 \end{enumerate}   
\end{lemma}
\begin{proof}
For every positive integers $m$ and $n$, we have  
\begin{multline*}
 | {\rm Cov} (X_0,X_k) | \leq \Vert X_k \Vert_{\infty} \Vert X_0 - \tX_{m,0}  \Vert_1+  | {\rm Cov} ( \tX_{m,0} ,X_k) | \\
  \leq \Vert X_k \Vert_{\infty} \Vert X_0 - \tX_{m,0}  \Vert_1 +  \Vert \tX_{m,0} \Vert_{\infty} \Vert X_k - \tX_{n,k}  \Vert_1  + 
   | {\rm Cov} ( \tX_{m,0} ,\tX_{n,k} ) | \, .
\end{multline*}
For every $k \geq 2$, we select $m= [k/2]$ and $n = [k/2] -1$. In this case $\tX_{m,0}$ and $\tX_{n,k}$ are independent, implying that for any $k \geq 2$, 
\[
 | {\rm Cov} (X_0,X_k) | \leq 
\Vert \psi \Vert_{\infty}  \big (  \Vert X_0 - \tX_{ [k/2] ,0}  \Vert_1  +  \Vert X_k - \tX_{ [k/2] -1,k}  \Vert_1 \big )  \, .
\]
 This upper bound combined with Proposition \ref{Ine410} proves the lemma. 
 \end{proof}

\subsection{Proof of  Proposition~\ref{thm:ASIPProp} when \texorpdfstring{$c^2=0$}{variance is zero}}
\label{Sectvariance is zero}

Like in the nonuniformly expanding case~\cite{CDKM18}, our strategy is to represent $X_k$ as a coboundary: $X_k = z_{k-1} - z_k$,
where $z_i$ is a sufficiently nice stationary sequence, and to bound the growth of $z_k$ almost surely.
Our proof follows well established ideas~\cite{B75,DMR24,MV16,S72} but unlike in the nonuniformly expanding case,
we have not found a result that we can conveniently cite.

The representation $X_k = z_{k-1} - z_k$ is provided by the following general result which is essentially due to Gordin \cite{Go73} (see  also 
~\cite[Lemma 7.1]{DMR24} in case of adapted r.v.'s but  there the condition c) is actually unnecessary). 
\begin{lemma}
    \label{Lemmadec}
    Suppose that $(X_k)_{k \in \Z}$ is a strictly stationary sequence of real-valued random variables that are in $L^2$,
    and let $({\cF}_k)_{k \in \Z}$ be a stationary filtration.
    Let $\E_i$ denote the conditional expectation given $\cF_i$ and set $S_n = X_1+ \cdots + X_n$.
    Assume that
    \begin{itemize}
        \item[(a)] $\E_0 (S_n)$ and $ \sum_{k =0}^{n} (  X_{-k} -  \E_{0} ( X_{-k})  )$ converge in $L^1$ as $n \to \infty$,
        \item[(b)] $\displaystyle  \liminf_{n \rightarrow \infty} \frac{\E |S_n| }{\sqrt{n}}=0$.
    \end{itemize}
    Then, almost surely,  $
        X_i = z_{i-1} - z_i $, 
    where $z_i = g_i-h_i$ with $g_i = \sum_{k \geq i  +1} \E_{i} ( X_k)$
    and $h_i = \sum_{k  \leq i } (  X_{k} -  \E_{i} ( X_{k}) )$.
    All $z_i, g_i, h_i$ are in $L^1$.
\end{lemma}

\begin{proof}
    Set $d_{i} = \sum_{k \in {\mathbb Z}}  P_i (X_{k+i}) $ where $P_i =  \E_i   - \E_{i-1} $.  Then
    \begin{equation}
        \label{deccobwithN}
        X_i = d_{i }  + z_{i-1} - z_{i}   . 
        .
    \end{equation}
    By assumption, all random variables in the above decomposition are in $L^1$.
    Moreover, $(d_{i})_{i \in \Z}$  is a stationary martingale difference sequence.
    Let's prove that the $d_i$'s are almost surely equal to zero.
    With this aim, we shall proceed as in~\cite{EJ85}. By Burkholder's deviation inequality
    for the martingale square function~\cite{Bu79} (see~\cite[Lemma~6]{EJ85} for an easy reference),
    there exists a positive constant $C$ such that, for any $\lambda >0$,
    \[
        \PP \Big (  n^{-1} \sum_{i=1}^n d_i^2 > \lambda^2 \Big )
        \leq C \lambda^{-1} \frac{ \E  \big | \sum_{i=1}^n d_i \big |}{ \sqrt{n}}
        .
    \]
  But, by \eqref{deccobwithN}, $ \E  \big | \sum_{i=1}^n d_i \big | \leq \E |S_n| + 2 \E |z_0| $. Using part~(b), there exists a subsequence $(n_k)_{k>0}$ tending to infinity such that $n_k^{-1} \sum_{i=1}^{n_k} d_i^2$ converges to zero in probability as $k \rightarrow \infty$. On the other hand, by the ergodic theorem, since  $(d_i^2)_{i \in {\mathbb Z}}$ is a stationary sequence of nonnegative r.v.'s, $n^{-1} \sum_{i=1}^n d_i^2$ converges almost surely to $\E (d_0^2 | {\mathcal I} ) $ where ${\mathcal I}$ is the so-called invariant $\sigma$-field. So, overall, $ \E (d_0^2 | {\mathcal I} ) =0$ almost surely and $\E(d_0^2) =  \E ( \E (d_0^2 | {\mathcal I} ))=0$. This proves that $d_0=0$ almost surely. 
\end{proof}

Now we return to the specific process we are interested in: $X_k = \psi ( (g_{k+\ell})_{\ell \in \Z})$ as in~\eqref{defprocessXk}.
Let $\cF_k = \sigma ( \eps_i, i \leq k)$.

\begin{prop}
    Suppose that $\E(T) < \infty$ and $c^2 = 0$. Then the assumptions of Lemma~\ref{Lemmadec} are satisfied.
\end{prop}

\begin{proof}
    Recall the construction of $X_{m,k}$ and 
    $\tX_{m,k} = \E ( X_{m,k} | \eps_{k-m}, \ldots, \eps_{k+m})$.
    For  $k \geq 0$, let  $X_{k}^*=\tX_{m_{k} ,k} $  with  $m_k  = \lfloor k/2 \rfloor$.
    Clearly $X_{k}^*$ is independent of $\cF_{0}$ and is centered.
    To prove the first part of assumption~(a), that is the convergence of $\E_0 (S_n)$, write
    \[
        \sum_{k \geq 2} \Vert \E_{0} ( X_k) \Vert_1
        \leq \sum_{k \geq 2} \Vert X_k - X_{k}^*  \Vert_1
        ,
    \]
    and use  Proposition~\ref{Ine410} to prove that the right hand side above converges.

    Next we prove the second part of assumption~(a).  For $k \geq 0$ let $X_{-k}^* = \tX_{\lfloor k/2 \rfloor , - k}$.
    Clearly $X_{-k}^*$ is $\cF_0$-measurable. Hence 
    \[
        \sum_{k  \geq 0} \Vert X_{-k} - \E_0 (X_{-k} )  \Vert_1
        = \sum_{k  \geq 0} \Vert X_{-k} - X_{-k}^* - \E_0 (X_{-k} - X_{-k}^* )  \Vert_1
        \leq 2 \sum_{k \geq 0} \Vert X_{-k} - X_{-k}^* \Vert_1
        . 
    \]
    By another application of Proposition~\ref{Ine410}, for any $r \geq 1$,
    \[
        \sum_{k  \geq 2} \Vert X_{-k} - \E_0 (X_{-k} )  \Vert_1
        \ll \sum_{k \geq 2}  \bigl( k^{-r/2} +  \PP( T \geq \lfloor k /(2 r) \rfloor )   \bigr)
        .
    \]
    Choosing $r >1$, the second part of assumption~(a) follows.

    We turn now to the proof of assumption~(b). By Lemma~\ref{lma33}, since $\E(T) < \infty$,
    $\sum_{k \geq 0} \bigl| {\rm Cov} (X_0, X_k)  \bigr| < \infty$ and then $c^2=\lim_{n \rightarrow \infty} n^{-1} \E(S_n^2)$ is well defined and is assumed to be zero. It follows that assumption~(b) is satisfied. 
    \end{proof}

At this point, as long as $c^2 = 0$, we have the representation
\[
    S_n
    = X_1 + \cdots + X_n
    = z_0 - z_n
\]
with $z_n$ as in Lemma~\ref{Lemmadec}.
Thus the proof of Proposition~\ref{thm:ASIPProp} is reduced to estimating $z_n$ almost surely,
and is completed by:

\begin{prop}
    \label{propcob}
    Suppose $\E(T) < \infty$ and $c^2 = 0$ and $z_i$ are as in Lemma~\ref{Lemmadec}.
    \begin{enumerate}
        \item If $\E (T^{p}) < \infty $ for some $p>1$, then
            $ \Vert z_0 \Vert_p  < \infty$ and $S_n = o (n^{1/(p+1)})$ a.s. 
        \item If $\E (\psi_p(T) )  < \infty $  with
            $\psi_p(x) = x^{p} ( \log x )^{ - ( \gamma + 1 + \eps)}$ for some $\eps >0$,
            then  $S_n = o (n^{1/(p+1)}  ( \log n)^{ ( \gamma + 1 + \eps) /(p+1) } )$ a.s.
        \item If $\PP( T \geq n) = O ( {\rm e}^{- \delta n^{\gamma}} )$ with some $\delta >0$ and $\gamma \in ]0,1]$,
            then $S_n = O (  ( \log n)^{1/ \gamma} )$ a.s.
    \end{enumerate}
\end{prop}

\begin{proof}
    Denote $M = \|X_0\|_\infty = \|\psi\|_\infty$.

    We begin with part 1: let $p > 1$ and suppose that $\E (T^p) < \infty$.
    By Proposition~\ref{Ine410},
    \begin{equation}
        \label{conditionforrate1p}
        \sum_{k >0} k^{p-1}  \bigl(
            \Vert X_k -   \tX_{ [k/2] ,k}  \Vert_1
            +  ( \Vert X_{-k} -   \tX_{ [k/2] ,-k}  \Vert_1 
        \bigr)
        < \infty
        , 
    \end{equation}
    First we use~\eqref{conditionforrate1p} to prove that $z_0$ is in $L^p$.
    Recall that $z_0 = g_0 - h_0$ where $g_0 = \sum_{k \geq 1} \E_{0} ( X_k)$
    and $h_0 =  \sum_{k  \leq 0 } (  X_{k} -  \E_{0} ( X_{k})  ) $.
    Observe that for $k \geq 2$, these summands can be bounded in $L^1$ by those in~\eqref{conditionforrate1p}:
    $\Vert \E_0 (X_k) \Vert_1 \leq \Vert X_k - \tX_{ [k/2] ,k} \Vert_1$ and 
    $\Vert X_{-k} -  \E_{0} ( X_{-k})  \Vert_1 \leq 2 \Vert X_{-k} - \tX_{ [k/2] ,-k}  \Vert_1$.
    Now, $g_0$, $h_0$ and consequently $z_0$ are in $L^p$ by an application of 
    Lemma~\ref{simple-lemma} with $V=g_0$ and $V=h_0$ respectively.
    (To prove that $g_0$ is in $L^p$ we can also use the arguments given in the proof of~\cite[Proposition~2.1]{DMR24}.) 

    To prove that $S_n = o (n^{1/(p+1)})$ a.s., it suffices to show that, for every $\eps >0$,  
    $ \sum_{\ell >0}  \PP \bigl( \max_{k \leq  2^{\ell}}  |S_k| > \eps 2^{\ell/(p+1)} \bigr) < \infty$,
    which is equivalent to proving that
    \begin{equation}
        \label{eq:bc}
        \sum_{n>1} n^{-1} \PP \bigl( \max_{k \leq  n}  |S_k| > \eps n^{1/(p+1)} \bigr)
        < \infty
        .
    \end{equation}

    Let $k \leq n$ and write
    \[
        S_k
        = S_k  - \E_k(S_k) + \E_k(S_k-S_n) +  \E_k(S_n)
        .
    \]
    Let $q \geq 1$.
    Using $\E_k( X_i  -  \E_{i-q} (X_i)  ) =0  $ if $i -q \geq k$, we have
    \begin{align*}
        \big | \E_k(S_n-S_k) \big |
        & \leq \Bigl| \sum_{i=k+1}^n \E_k( X_i  -  \E_{i-q} (X_i)  ) \Bigr|
        + \Bigl| \sum_{i=k+1}^n \E_k(   \E_{i-q} (X_i)  ) \Bigr|
        \\
        & \leq 2qM +  \sum_{i=k+1}^n  | \E_k ( \E_{i-q} (X_i) |
        .
    \end{align*}
    Then, using $\E_k(  \E_{\ell + q } (X_{\ell } ) ) =  \E_{\ell + q } (X_{\ell } ) $
    for $\ell \leq k-q$,
    \begin{align*}
        |S_k  - \E_k(S_k)|
        & \leq \Bigl| \sum_{\ell=1}^{k - q} \bigl(
            X_{\ell } - \E_{\ell + q} X_\ell
            - \E_k( X_{\ell } - \E_{\ell + q} X_\ell )
        \bigr) \Bigr|
        + 2 q M
        \\
        & \leq
        \sum_{\ell=1}^{k - q}  \vert X_{\ell } - \E_{\ell + q } (X_{\ell } ) \vert
        +  \sum_{\ell=1}^{k - q}   \E_k \vert X_{\ell } - \E_{\ell + q } (X_{\ell } ) \vert
        + 2 q M
        .
    \end{align*}
    Overall, 
    \begin{multline*}
        \max_{1 \leq k \leq n} |  S_k  |
        \leq 4 q M
        + \max_{1 \leq k \leq n}  |  \E_k ( S_n)   |
        + \max_{1 \leq k \leq n}   \sum_{i=1}^n  | \E_k ( \E_{i-q} (X_i) |
        \\
        + \sum_{\ell=1}^{n}  \vert X_{\ell } - \E_{\ell + q } (X_{\ell } ) \vert
        +  \max_{1 \leq k \leq n}  \sum_{\ell=1}^{n} \E_k \vert X_{\ell } - \E_{\ell + q } (X_{\ell } ) \vert
        .
    \end{multline*}
    Accordingly, if $x > 0$ and $qM \leq x$, then
    \begin{equation}
        \label{decofmax}
        \begin{multlined}
            \PP \bigl( \max_{1 \leq k \leq n} |S_k| > 8x \bigr)
            \leq \PP \bigl( \max_{1 \leq k \leq n}  |  \E_k ( S_n)   |> x \bigr)
            + \PP \Bigl( \max_{1 \leq k \leq n} \sum_{i=1}^n  | \E_k ( \E_{i-q} (X_i) |> x \Bigr)
            \\
            + \PP \Bigl(  \sum_{\ell=1}^{n}  \vert X_{\ell } - \E_{\ell + q } (X_{\ell } ) \vert >x \Bigr)
            +  \PP \Bigl(   \max_{1 \leq k \leq n}  \sum_{\ell=1}^{n} \E_k \vert X_{\ell } - \E_{\ell + q } (X_{\ell } ) \vert>x \Bigr)
            .
        \end{multlined}
    \end{equation}

    Let
    \[
        \theta(q) =   \Vert X_0 -   \tX_{q ,0}  \Vert_1
        .
    \]
    Starting from \eqref{decofmax} and using Doob's maximal inequality, as done in the proof of inequality (A.42) in \cite{MR12},
    we infer  that for any nondecreasing, non negative and convex function $\varphi$ and any $x > 0$,
    \begin{equation}
        \label{inemax}
        \PP \Bigl( \max_{1 \leq k \leq n} |S_k| > 8x \Bigr)
        \ll \frac{\E ( \varphi (S_n) ) }{ \varphi (x) } + n x^{-1} \theta(\lfloor x/M \rfloor)
        .
    \end{equation}
    Then, taking $\varphi(x) =x^p$,
    \begin{align*}
        \sum_{n>1} n^{-1} \PP \bigl( \max_{k \leq  n}  |S_k| > \eps n^{1/(p+1)} \bigr)
        & \ll \sum_{n > 1} n^{-1-\frac{p}{p+1}} \Vert S_n \Vert_p^p
        + \sum_{n > 1} n^{-1/(p+1)} \theta(\lfloor \eps n^{1/(p+1)} / M \rfloor)
        \\
        & \ll \sup_n \Vert S_n \Vert_p^p + \sum_{n >1} n^{p-1} \theta (n)
        .
    \end{align*}
    The expression above is finite: $\|S_n\|_p \leq 2 \|z_0\|_p < \infty$ and
    $\sum_{n >1} n^{p-1} \theta (n) < \infty$ by Proposition~\ref{Ine410}.
    This verifies~\eqref{eq:bc} and completes the proof of part~1.

    We now turn to part~2. By part 1, $\Vert S_n \Vert_r < \infty$ for all $r \in [1, p[$.
    Hence, with $b_n = n^{1/(p+1)} (\log n)^{  ( \eps  +  \gamma +1 )/(p+1) }$,
    applying~\eqref{inemax} with $\varphi(x) = x^r$, it follows that for any $\eta >0$, 
    \[
        \PP ( \max_{1 \leq k \leq n} |S_k| > \eta b_n)
        \ll  b_n^{-r}  + \min \Bigl( 1, \frac{n}{ \eta b_n}   \PP (T >c \eta b_n)  \Bigr)
        ,
    \]
    where $c$ is a positive constant which does not depend on $n$.
    Part~2 is proved if we show that $\sum_{n \geq 1} b_n^{-1} \PP (T > c \eta  b_n)  < \infty$.
    By a change of variable this is equivalent to
    $\sum_{n \geq 2} n^{p-1} (\log n )^{  \eps  + \gamma +1} \PP (T > n) < \infty$
    which holds by Proposition~\ref{Ine410} because $ \E (\psi_p (T)) < \infty$.
    
    We turn now to the proof of part~3. 
    Let $K>0$, to be chosen later.  By~\eqref{inemax} applied with with
    $x= K(\log n)^{1/\gamma}$ and $\varphi(u)=u^{ 2}$, using
    Lemma~\eqref{simple-lemma}, our assumption and part~2 of Proposition~\ref{Ine410}, we see that
    \begin{align*}
        \frac{1}{n} \PP \bigl( \max_{1\le k\le n} |S_k|>8K(\log n)^{1/\gamma} \bigr)
        & \ll 
        \frac{1}{n (\log n)^{2/\gamma}} + \frac1{K(\log n)^{1/\gamma}} {\rm e}^{-\delta' (K(\log n)^{1/\gamma}/M)^\gamma}
        \\ 
        & \ll  \frac1{n (\log n)^2}
        +\frac1{n^{\delta'(K/M)^\gamma}}
        .
    \end{align*}
    We conclude by taking $K =2M\delta'^{-1/\gamma}$.
\end{proof}

\begin{lemma}\label{simple-lemma}
    Let $(V_k)_{k\in \N}$ be a sequence of non negative random variables uniformly bounded by $M$.
    Set $V = \sum_{k\in \N}V_k$. Then, for every $p>1$, we have 
    \begin{equation}
        \|V\|_p
        \le M+ \sum_{k\in \N} k^{p-1}\|V_k\|_1
        .
    \end{equation}
\end{lemma}

\begin{proof}
    Recall that $\|V\|_p=\sup_{\|Z\|_q=1}\E(ZV)$, where $q=p/(p-1)$. Let $Z$ be a random variable with $\|Z\|_q=1$.
    Then 
    \begin{align*}
        \E(|ZV|)
        &   \le \sum_{k\in \N} k^{p-1}\E (|V_k|) + \sum_{k\in \N} M \E \bigl( |Z| {\bf 1}_{\{|Z|\ge k^{p-1}\}} \bigr)
        \\
        &  \qquad =  \sum_{k\in \N} k^{p-1}\E(|V_k|) + M\E\Bigl( |Z|\sum_{1\le k\le |Z|^{1/(p-1)}} 1 \Bigr)
        \\
        &  \qquad = \sum_{k\in \N} k^{p-1}\E(|V_k|) + M\E(|Z^q|)\, ,
    \end{align*} 
    and the lemma is proved.
\end{proof}

\subsection*{Acknowledgements}
The authors are grateful to E.~Rio for useful discussions concerning the proof
of Lemma~\ref{lem:MTB1}.
The authors are also grateful to the three referees for their careful reading and comments,
which improved the presentation of the paper.   

\subsection*{Statements and declarations}
\noindent \textbf{Funding.}
A.K.~has been partially supported by
the European Research Council (ERC) under the European Union's Horizon 2020
research and innovation programme (grant agreement No 787304)
and by EPSRC grant EP/V053493/1.

\noindent \textbf{Conflict of interest.} The authors declare that they have no financial interests.

\noindent \textbf{Data availability.} There is no associated data.

\end{document}